\documentclass[12pt, fceqn]{article}
\usepackage{hyperref}
\hypersetup{
	bookmarks=true,         
	unicode=false,          
	pdftoolbar=true,        
	pdfmenubar=true,        
	pdffitwindow=false,     
	pdftitle={Certificate},    
	pdfauthor={Dr. Harish Kumar},     
	pdfsubject={TEQIP certificates},   
	pdfcreator={Dr. Harish Kumar},   
	pdfproducer={},  
	pdfkeywords={Certificates,} {TEQIP} {Participation}, 
	pdfnewwindow=true,      
	colorlinks=false,       
	linkcolor=red,          
	citecolor=green,        
	filecolor=magenta,      
	urlcolor=cyan           
}
\usepackage{amsfonts, amsthm, latexsym, amssymb}
\newcommand{\BigO}[1]{\ensuremath{\operatorname{O}\bigl(#1\bigr)}}
\usepackage[fleqn]{amsmath}
\usepackage{graphicx}
\usepackage{graphics}
\usepackage{color}
\usepackage{comment}
\usepackage{algorithm}
\usepackage{algorithmic}
\usepackage{lscape}
\usepackage{footnote}

\usepackage{epsfig}
\usepackage{fancybox}
\usepackage{psfrag}
\usepackage{tabularx}
\usepackage{float}
\usepackage{multirow}
\usepackage{rotating}
\usepackage{cancel}
\usepackage{epstopdf}
\usepackage{booktabs}

\date{}

\DeclareMathOperator*{\Min}{\mbox{Min }}
\DeclareMathOperator*{\Max}{\mbox{Max }}

\DeclareMathOperator*{\Argmin}{\mbox{Argmin}}
\DeclareMathOperator*{\argmin}{\mbox{argmin}}

\newcommand{\om}{\omega}
\renewcommand{\O}{\Omega}

\newtheorem{thm}{THEOREM}[section]
\newtheorem{cor}[thm]{COROLLARY}
\newtheorem{lem}[thm]{LEMMA}
\newtheorem{prop}[thm]{PROPOSITION}
\theoremstyle{remark}

\theoremstyle{definition}
\newtheorem{defn}[thm]{DEFINITION}

\renewcommand{\O}{\Omega}

  {\begin{list}{}{\setlength{\itemsep}{4pt}
  \parindent 0pt  \setlength{\parsep}{0pt}\setlength{\leftmargin}{+25pt}
  \setlength{\itemindent}{-\parindent}}}{\end{list}}

\begin{document}
\begin{center}
{\Large  \textbf{Integer Set Reduction for Stochastic Mixed-Integer Programming}}\\[12pt]


\mbox{\large Saravanan Venkatachalam and Lewis Ntaimo}\\
Department of Industrial and Systems Engineering, Texas A\&M
University, 3131 TAMU, \\ College Station, TX 77843, USA. \\
\mbox{saravanan@tamu.edu,ntaimo@tamu.edu}\\[8pt]

\normalsize
\end{center}

\begin{abstract}
Two-stage stochastic mixed-integer programming (SMIP) problems with general integer variables in the second-stage are generally difficult to solve. This paper develops the theory of integer set reduction for characterizing the subset of the convex hull of feasible integer points of the second-stage subproblem which can be used for solving the SMIP. The basic idea is to consider a small enough subset of feasible integer points that is necessary for generating a valid inequality for the integer subproblem. An algorithm for obtaining such a subset based on the solution of the subproblem LP-relaxation is then devised and incorporated into the Fenchel decomposition method for SMIP. To demonstrate the performance of the new integer set reduction methodology, a computational study based on randomly generated test instances was performed. The results of the study show that integer set reduction provides significant gains in terms of generating cuts faster leading to better bounds in solving SMIPs than using a direct solver.

\end{abstract}

\section{Introduction}\label{sec-intro}
A two-stage stochastic mixed-integer programming (SMIP) problem involves optimizing the here-and-now (first-stage) costs plus expected future (second-stage) costs. Solving SMIP is still challenging and this paper makes strides towards that by introducing the theory of \emph{integer set reduction} for characterizing \emph{subsets} of the convex hull of feasible integer points of the second-stage subproblem that can be used to generate valid inequalities (cutting planes or cuts) for SMIP. The goal of integer set reduction is to speed up the cut generation routines and potentially lead to faster solution times for SMIP than direct solvers. In this paper we consider SMIP problems of the following form:

\begin{equation}
\begin{alignedat}{2}\label{eq-1a}
\text{SIP2: } \Max \,\,  & c^\top x + \mathcal{Q}_{E}(x)		   && \\
\text{s.t. }        & A x \leq b                             && \\
& x \in X.                               &&
\end{alignedat}
\end{equation}

\noindent In problem SIP2, $x \in \mathbb{R}_+^{n_1}$ denotes the first-stage decision vector, $c \in \mathbb{R}^{n_1}$ is the first-stage cost vector, $b \in \mathbb{R}^{m_1}$ is the first-stage right hand side, and $A \in \mathbb{R}^{m_1 \times n_1}$ is the first-stage constraint matrix. The set $X$ imposes binary restrictions on all or some components of $x$. The function $\mathcal{Q}_{E}(x)$ denotes the expected second-stage cost based on $x$. The function $\mathcal{Q}_{E}(x)$ is the expected recourse function and is given as follows:
\begin{equation}\label{eqn2}
\mathcal{Q}_{E}(x) = \mathbb{E}_\omega \Phi(q(\om),h(\om)-T(\om)x,\om),
\end{equation}
where $q(\omega) \in \mathbb{R}^{n_2}$ is the cost vector, $h(\omega) \in \mathbb{R}^{m_2}$ is the right hand side vector, and $T(\omega) \in  \mathbb{R}^{m_2 \times n_1}$ is the technology matrix. The second-stage function $\Phi$ is a value function of a mixed-integer program (MIP) and is given as follows:
\begin{equation}\label{eqn2a}
\Phi(\rho,\tau,\om) = \Max \{\rho^\top y(\om): Wy(\om) \leq \tau, 0 \leq y(\om) \leq u,  y(\om) \in Y\}.
\end{equation}
In the second-stage (scenario) problem \eqref{eqn2a}, $y(\om)$ denotes the recourse decision vector and $W \in \mathbb{R}_+^{m_2 \times n_2}$ is the fixed recourse matrix. The vector $u \in \mathbb{Z}_+^{n_2}$ is the upper bound on the second-stage decision variables. It is assumed that $T(\om):\Omega \mapsto \mathbb{R}^{m_2 \times n_1} $, $h(\om):\Omega \mapsto \mathbb{R}^{m_2}$ and $q(\om):\Omega \mapsto \mathbb{R}^{n_2}$ are measurable mappings defined on a probability space $(\Omega,\mathcal{F},\mathbb{P})$. The set $Y$ imposes integer restrictions on all or some components of $y(\om)$. The function $\mathcal{Q}_{E}(x)$ is the expected recourse function, where $\omega$ is a realization of a multivariate random variable $\tilde{\omega}$, and $\mathbb{E}_\omega$ denotes the mathematical expectation operator.

We consider problem SIP2 under the following assumptions:
\begin{itemize}
\item[\textbf{A1.}]
The random variable $\tilde{\om}$ is discrete with finitely many scenarios $\om \in \O$, each with probability of occurrence $p(\om)$ such that $\sum_{\om \in \O}p_{\omega} = 1$.
\item[\textbf{A2.}] The first-stage feasible set $\{A x \leq b, x \in X\}$ is nonempty.
\item[\textbf{A3.}] The right hand side vector $\tau$ and fixed-recourse matrix $W$ are nonnegative, and $W$ is rational.
\item[\textbf{A4.}] The second-stage feasible set $\{W y(\om) \leq \tau, 0 \leq y(\om) \leq u, y(\om) \in Y\}$ and is bounded and nonempty for all $x \in \{A x \leq b, x \in X\}$.
\end{itemize}

\noindent Assumption (A1) is needed for tractability while assumptions (A2) and (A4) are needed to guarantee that the problem has an optimal solution. Assumption (A3) and (A4) are needed for the proposed integer set reduction method to allow for a well-defined problem and finite convergence of the cutting method. Assumption (A4) implies the relatively recourse assumption, i.e., $\mathbb{E}_\omega [\,| \Phi(q(\om), h(\om)-T(\om)x),\omega |\,] < \infty$ for all $x \in \{A x \leq b, x \in X\}$. Because of assumption (A1), SIP2 can be written in extensive form as a so-called deterministic equivalent problem (DEP) as follows:
\begin{equation}
\begin{alignedat}{2}\label{eq-SIP2-dep}
\text{DEP}: \Max \,\, & c^\top x + \sum_{\omega \in \Omega} p_\omega q(\omega)^\top y(\omega)  && \\
\text{s.t. }        & A x \leq b                             && \\
&          T(\omega) x + W y(\omega) \leq h(\omega)         && \\
& x \in X, y(\omega) \in Y.               &&
\end{alignedat}
\end{equation}

In the above formulation, $p_\omega$ denotes the probability of occurrence for the scenario $\omega$, and $\sum{p_{\omega \in \Omega}} = 1$. Even for a reasonable number of scenarios in $\Omega$, DEP is a large-scale MIP. With integer variables in both first and second-stages, a moderate sized DEP may be difficult to solve using a direct solver such as CPLEX \cite{CPLEX}. This makes a decomposition approach a necessity for most practical sized problems. In SIP2, the type of decision variables (continuous, binary, integer) and in which stage they appear greatly influences algorithm design. The complexity of the solution method depends on the definitions of the sets $X$ and $Y$. When both these sets do not impose integer restrictions on the decision variables, the recourse function $\Phi(\rho,\tau,\om)$ is a well-behaved piecewise linear and convex function of $x$. Thus, Benders' decomposition \cite{benders1962partitioning} is applicable in this case \cite{wollmer1980two} and the L-shaped method \cite{van1969shaped} can be used to solve the problems. Assuming fixed recourse (i.e, the recourse matrix $W$ is independent of $\omega$),  the value function of $\Phi(\rho,\tau,\om)$ is a piecewise linear function in $x$. Hence, the L-shaped method works by approximating the linear functions from the subproblems by constructing optimality cuts in the first-stage based on the dual values from the subproblems. However, when $Y \in \mathbb{Z}^+$, the linear approximation procedure by L-shaped method is not viable, as the value function is generally discontinuous and is lower semicontinuous \cite{blair1982value}. Also, the function is non-convex and sub-additive \cite{schultz1993continuity}.
Hence, new algorithms or extensions of the L-shaped method are required to handle integer variables in the second or in both of the stages.

Cutting plane methods that can partially approximate the second-stage problems within the L-shaped method have been proposed for SMIPs with integer variables in the second-stage. In \cite{caroe1997cutting}, a lift-and-project cutting plane approach based on the ideas from \cite{balas1993lift} is used to solve problems with binary and continuous variables in both the first- and second-stage.
In \cite{sen2005c}, for problems with binary variables in first-stage, and binary and continuous variables in second-stage, disjunctive cuts are developed for the second-stage. The work in \cite{sherali1998reformulation} and \cite{sherali2002modification} uses the framework of reformulation linearization technique. The algorithm in \cite{sen2005c} is extended in \cite{sen2006decomposition} for problems with binary, continuous and discrete variables in the second-stage. Applications of SMIP in supply chain, air traffic control, and auto-carrier vehicle loading problems can be found at \cite{beier2015nodal}, \cite{corolli2015two}, and \cite{venkatachalam2014algorithms}, respectively.

Fenchel cuts are suggested in \cite{boyd1994fenchel}, and a number of characteristics are derived in \cite{boyd1993solving},  \cite{boyd1994solving} and \cite{boyd1995convergence}. The most important results from \cite{boyd1994fenchel}, \cite{boyd1994solving} and \cite{boyd1995convergence} are that Fenchel cutting planes are facet defining under certain conditions, and the use of Fenchel cuts in a cutting plane approach yields an algorithm with finite convergence. The work also highlights the fact that generating a Fenchel cut for binary programs is computationally expensive in general; therefore, problems with special structure are desirable to achieve faster convergence. Computational experiments demonstrating the effectiveness of Fenchel cuts  are presented for knapsack polyhedra in \cite{boyd1993generating} and for pure binary problems in \cite{boyd1994solving}.

Since the pioneering work in \cite{boyd1994fenchel}, only a few works in the literature have adopted Fenchel cuts. In \cite{saez2000solving}, Fenchel cuts are used to improve the bounds obtained from MIPs using Lagrangian relaxation. Fenchel cuts are used to solve deterministic capacitated facility location problems \cite{ramos2005solving}. This work compares Fenchel cuts to Lagrangian cuts in finding good relaxation bounds for their problem. In \cite{boccia2008cut}, Fenchel cutting planes are used for finding $p$ median nodes in a graph using a cut and branch approach. Fenchel cuts are first derived for two-stage SMIPs under a stage-wise decomposition setting in \cite{ntaimo2013fenchel} and are referred to as Fenchel decomposition (FD) cuts. Extensive study of FD cuts for two-stage SMIP with binary decision variables in both first-stage and second-stages are given in \cite{FDDECOMP} and \cite{venkatachalam2014algorithms}.

This work makes the following contributions to the literature on stochastic programming:  (a) deriving \emph{integer set reduction} theory for determining subsets of the second-stage feasible integer set to use for faster cut generation; (b) devising an algorithm for obtaining such subsets based on the solution of the subproblem LP-relaxation; (c) applying the integer set reduction in the context of FD cuts for solving SMIPs with general integer variables in the second-stage; and (d) reporting on a computational study that demonstrates the advantages of integer set reduction. In the literature, Fenchel cuts and FD cuts are derived and used for MIP and SMIP, respectively, with binary variables. This work is the first to derive these cuts for MIP and SMIP with \emph{general} integer variables.

The rest of this paper is organized as follows: Integer set reduction theory is derived in Section 2 and is applied to the FD setting for two-stage SMIP in Section 3. A computational study to illustrate the performance of the new methodology is reported in Section 4. Finally, conclusions and directions for future research are given in Section 5.

\section{Integer Set Reduction for Cut Generation}
In this section we develop the integer set reduction theory to characterize the properties of the convex hull of integer points needed for generating a valid inequality for the second-stage feasible set. We then use the theory to devise an algorithm for obtaining a reduced set of integer points needed for generating a valid inequality based on the second-stage LP-relaxation. To illustrate the concepts, we use simple numerical examples. Next we start with some preliminaries.

\subsection{Preliminaries}
We will now provide some important definitions needed in the derivation of a valid inequality for SIP2 with general integer variables in the second-stage. Suppressing $\om$ for simplicity in exposition, the feasible set for the second-stage subproblem \eqref{eqn2a} can be given as follows:
\begin{equation}\label{eqn5}
F^{IP} = \{y:Wy \leq \tau, 0 \leq y \leq u, y \in Y \}.
\end{equation}
Thus subproblem \eqref{eqn2a} can now be rewritten as
\begin{equation}\label{eqn2aIP}
\Max \{\rho^\top y: y \in F^{IP}\}.
\end{equation}
The LP-relaxation feasible set to $F^{IP}$ can be expressed as
\begin{equation}
F^{LP} = \{y: Wy \leq \tau, 0 \leq y \leq u, y \in \mathbb{R}_+^{n_2} \} \notag
\end{equation}
and the LP-relaxation to \eqref{eqn2aIP} given as
\begin{equation}\label{eqn2aLP}
\Max \{\rho^\top y: y \in F^{LP}\}.
\end{equation}
We shall denote by $C(F^{IP})$ the convex hull of integer points in $F^{IP}$. Let $F^{IP}_{\upsilon} \subseteq F^{IP}$ and  $C(F^{IP}_{\upsilon})$ will denote the convex hull of integer points in $F^{IP}_{\upsilon}$.

Now let $\hat{y} \in F^{LP}$ be the optimal solution to subproblem \eqref{eqn2aLP} for a given $x \in \{A x \leq b, x \in X\}$. Our goal is to use the point $\hat{y}$ and restrict the derivation of valid inequalities (cuts) to a relatively small subset of integer points $F^{IP}_{\upsilon}$ instead of $F^{IP}$ so that a generated cut valid for $C(F^{IP}_{\upsilon})$ is valid for $C(F^{IP})$ and cuts off $\hat{y}$. The hope is that doing so would result in reduced cut generation computational time, thus leading to fast cutting plane methods for SIP2. Therefore, it is desirable to have $F^{IP}_{\upsilon} \subset F^{IP}$ such that
$|F^{IP}_{\upsilon}| << |F^{IP}|$ so that generating cuts over $F^{IP}_{\upsilon}$ is not expensive.

\begin{defn}\label{defn1}
An inequality is said to be \textit{valid} for the set $C(F^{IP})$ if it is satisfied by every point in the set. A cut with respect to a point $\hat{y} \notin C(F^{IP})$ is a valid inequality for $C(F^{IP})$ that is violated by $\hat{y}$.
\end{defn}

\begin{figure}[H]
	\centering
	\includegraphics[scale=1.00]{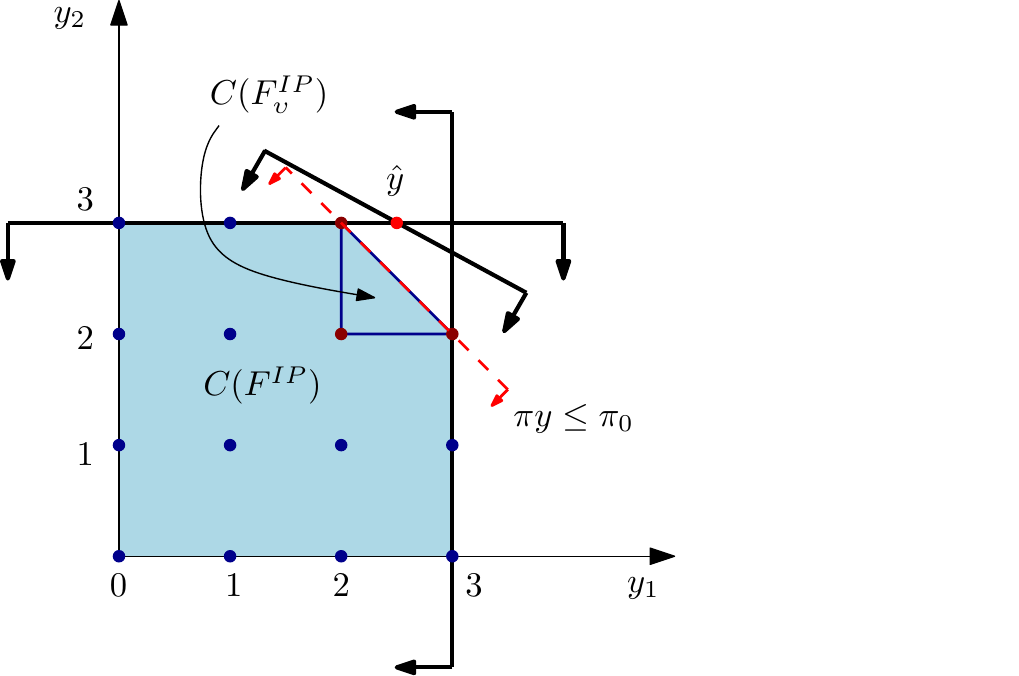}\\
	\caption{Separation problem with reduced integer feasible set}
	\label{sep-fig2}
\end{figure}

Generating a valid inequality using the subset $F^{IP}_{\upsilon}$ is depicted in Figure \ref{sep-fig2}. In the figure, given $\hat{y} \in F^{LP}$, the three points defining the triangle constitute $F^{IP}_{\upsilon}$ and are used to generate the cut (dashed lines) of the form $\pi^\top y \leq \pi_0$. We devise a methodology to obtain  $F_{\upsilon}^{IP}$, and subsequently use it to generate a valid inequality.  Also, since $F_{\upsilon}^{IP} \subseteq F^{IP}$, we need to form $F_{\upsilon}^{IP}$ such that the generated valid inequality does not cut off any integer points in $F^{IP}$.

In subproblem \eqref{eqn2a}, the general integer variable vector $y$ is bounded by the vector $u$. Let $I$ be the set of indices of the components of $y$ which are integer variables, and $K$ be the set of indices of the constraints in \eqref{eqn2a}. Also, let the elements of the matrix $W$ be denoted by $w_{kt}$, where $k \in K$ is the constraint index, and $t \in I$ is the decision variable index. We also make the following assumption regarding $W$:

\begin{itemize}
\item[\textbf{A5.}] The polytope defined by $C(F^{IP})$ is assumed to be full dimensional with dimension $n_2$.
\end{itemize}

Let $P$ be an index set for integer points in $F^{IP}$ such that for $p \in P$ we have an integer point $y^p \in F^{IP}$. The $i^{th}$ component of $y^p$ will be denoted $y_i^p$. In Figure \ref{sep-fig2}, for example, the integer point $y^p = (3, 2)$ has components $y_1^p=3$ and $y_2^p=2$. Let $y^{IP}$ be the optimal solution to \eqref{eqn2a} and $\hat{y}$ be the optimal solution to the LP-relaxation. Also, define $\bar{y} = \lfloor \hat{y}\rfloor$ so that the components $\bar{y}_i$, for all $i = 1, \cdots, n_2$ are initialized as $\bar{y}_i = \lfloor \hat{y}_i\rfloor$. Let  $d_{ij}^p$ denote the distance from $y^p$ to the boundary of $F^{LP}$ along the $y^p_j$ axis for all $i, j \in I, i \neq j$. Then $d_{ij}^p$  can be calculated as follows:

\begin{equation}\label{dval}
\begin{alignedat}{2}
d_{ij}^{p} = \min \left\{ \left( \left[ \tau_k - \sum_{t \in I: t \neq i, t \neq j}w_{kt}\hat{y}_t -w_{ki}\bar{y}_i \right] /  w_{kj} \right)_{\forall k \in K}  - y_j^p, \, \, u_j - y_j^p \right\} ,
\end{alignedat}
\end{equation}
where $\tau_k$ is the right hand side for constraint $k \in K$, and $u_j$ is the upper bound for axis $j$. For each $y^p_i$ axis, $i =1,\cdots,n_2$, let $f_i(y^p)$ be the shortest distance from $y^p$ to the boundary of $F^{LP}$ along the $y^p_j$ axis, $i \neq j$.  Then $f_i(y^p)$ can be calculated as follows:
\begin{equation}\label{dval1}
\begin{alignedat}{2}
f_i(y^p) = \min_{j \in I, \ i \neq j} \left\lbrace d_{ij}^{p}  \right\rbrace.
\end{alignedat}
\end{equation}
For each $y^p_i$ axis, $i =1,\cdots,n_2$, let $f'_i(y^p)$ be the shortest distance from $y^p$ to the boundary of $F^{LP}$ along the $y^p_i$ axis.  Then $f'_i(y^p)$ can be calculated as follows:
\begin{equation}\label{dval2}
\begin{alignedat}{2}
f'_i(y^p) = \min_{j \in I, \ i \neq j} \left\lbrace d_{ji}^{p}  \right\rbrace.
\end{alignedat}
\end{equation}

 Next, we establish that there exists a set $F_{\upsilon}^{IP} \subseteq F^{IP}$, which is sufficient for generating a valid inequality for $F^{IP}$.
\\

\begin{lem}\label{lemma1}
	Given $F^{IP}$, there exists a set $F_{\upsilon}^{IP} \subseteq F^{IP}$ such that $F_{\upsilon}^{IP}$ is sufficient for generating a valid inequality  for $F^{IP}$.
\end{lem}
\begin{proof}
	Given $F^{IP}$, then either $F_{\upsilon}^{IP} = F^{IP}$ or $F_{\upsilon}^{IP} \subset F^{IP}$. For $F_{\upsilon}^{IP} = F^{IP}$, it is obvious that entire set $F^{IP}$ can be used for generating a valid inequality to cut off a fractional point $\hat{y}$. Now consider the case $F_{\upsilon}^{IP} \subset F^{IP}$. $\vert F^{IP} \vert \geq n_2+1$ as $C(F^{IP})$ is full dimensional by assumption (A5). To generate a valid inequality, $\vert F_{\upsilon}^{IP} \vert \geq n_2$, as $n_2$ affinely independent integer points are needed to construct a facet for $C(F^{IP})$. Then $F_{\upsilon}^{IP}$ can be constructed such that there exists an integer point $\, \, y^{p'} \notin F_{\upsilon}^{IP}$ and $ y^{p'} \in F^{IP}$. Hence, $\vert F^{IP} \vert > \vert F_{\upsilon}^{IP} \vert$ which gives $F_{\upsilon}^{IP} \subset F^{IP}.$
\end{proof}
\bigskip

We state the required properties for an integer point $y^p \in F_{\upsilon}^{IP}$ in the following corollary.
\\

\begin{cor}\label{corr1}
	Let $y^p \in F_{\upsilon}^{IP}$. Then either of the following must be true:
	\begin{itemize}
		\item[(i)]  $d_{ij}^{p} < 1$, $\forall i,j \in I \vert j \neq i$ or
		\item[(ii)] $d_{ij}^{p} \geq 1$ and any integer point $y^{p'} \in F^{IP}$ such that $y_i^{p'} -  y_i^{p} \geq 1, \forall i \in I$ for at least one index $i \in I$ also belongs to $F_{\upsilon}^{IP}$.
	\end{itemize}
\end{cor}

\noindent When $d_{ij}^{p} < 1$, then there does not exist an integer point$\,\,\,\,  y_i^{p} \in F^{IP}$. Alternatively, when $d_{ij}^{p}  \geq 1$, then there exists an integer point   $ \,\,\, y_i^{p'} \in F^{IP}$ such that $d_{ij}^{p}  > d_{ij}^{p'} $. This means that there is an integer point $y_i^{p'}$ between $y_i^{p}$ and the boundary of $F^{LP}$. Since $y_i^{p} < y_i^{p'}$ and $y_i^{p}, y_i^{p'} > 0$, it implies that $d_{ij}^{p} - y_i^{p} > d_{ij}^{p} - y_i^{p'}$ and $y_i^{p'} - y_i^{p} > 0 $. This means that $y_i^{p'} - y_i^{p} \geq 1 $ since $y_i^{p'},y_i^{p} \in F^{IP}$. However by \textit{(ii)}, if $d_{ij}^{p} \geq 1$ and $y_i^{p'} -  y_i^{p} \geq 1, \forall i \in I$, then the point $y^{p'} \in F_{\upsilon}^{IP}$.

In Corollary \ref{corr1}, we simply state the properties for the integer points that define the set $F_{\upsilon}^{IP}$. However, it is desirable to get the smallest possible set $F_{\upsilon}^{IP}$ such that the valid inequality generated based on $F_{\upsilon}^{IP}$ does not cut off any (optimal) integer point in $F^{IP}$. We evaluate each of the components of $y$ and add an integer point $y^p$ to the set $F_{\upsilon}^{IP}$ if it is the closest integer point to $\hat{y}$ for that component, or if all other integer points between $y^p$ and the boundary of $F^{LP}$ are already in the set $F_{\upsilon}^{IP}$.

In the following lemma, we state the requirements for the minimum cardinality for the set $F_{\upsilon}^{IP}$. Ideally, we would like the set $F_{\upsilon}^{IP}$ to have a small number of integer points since generating a cut based on $C(F^{IP})$ may be computationally expensive.
\\

\begin{lem}\label{lemma2}
	Given $F_{\upsilon}^{IP}$ and the set $D_i = \Argmin_{y^p} \ \lbrace f_i(y^p) \mid y^p \in F^{IP}, y_i^p > 0 \rbrace$ there exists an integer point $y' \in F_{\upsilon}^{IP}$ for every $i \in I$ such that
	
	\begin{equation}\label{p2eq}
	\begin{alignedat}{2}
	y' = \argmin_{y^k} \,\,\, \lbrace f'_i(y^k) \mid y^k \in D_i \rbrace. \nonumber
	\end{alignedat}
	\end{equation}
\end{lem}

\begin{proof}
	We prove this result by contradiction. Suppose that the point $y' \in F_{\upsilon}^{IP} $ does not exist. Then this means that a valid inequality would pass through the origin. This implies that $\, \, C(F^{IP})$ is not full dimensional, which is a contradiction due to  assumption (A5). Hence, there exists a point $y'$ for each axis $i$.
\end{proof}

\begin{prop}\label{prop1}
	The minimum cardinality of the set $F_{\upsilon}^{IP}$ is $n_2$.
\end{prop}

\begin{proof}	
	By Lemma \ref{lemma2}, there should be at least one integer point for every component $i \in I$. This implies that $\left\vert I \right\vert=n_2$. Also, since valid inequalities are facets, then we need at least $n_2$ affinely independent points in the set $F_{\upsilon}^{IP}$ for generating the facet.
\end{proof}

\subsection{Integer Set Generation Algorithm}

Based on Lemma \ref{lemma2} and Proposition \ref{prop1}, $\left\vert{F_{\upsilon}^{IP}}\right\vert = n_2$ since subproblem \eqref{eqn2a} is full dimensional. However, getting the smallest set is not trivial unless we have an oracle providing an ideal interior point $y^p \in F^{IP}$, on which the smallest set $F_{\upsilon}^{IP}$ can be constructed. In the next section, we devise an algorithm to obtain the set $F_{\upsilon}^{IP}$ using Corollary \ref{corr1}, Lemma \ref{lemma2}, and Proposition \ref{prop1}. In the algorithm, we start with an initial point $y^p \in F_{\upsilon}^{IP}$, where $y^p$ is constructed based on $\hat{y}$. Corollary \ref{corr1} and Lemma \ref{lemma2} are used to check whether the set of points in $F_{\upsilon}^{IP}$ are sufficient for generating a valid inequality, if not then the set $F_{\upsilon}^{IP}$ is expanded by sequentially adding integer points from the set $F^{IP}$. An algorithm for obtaining the set $F_{\upsilon}^{IP}$ can be stated as follows:

\begin{algorithm}[H]
	\caption{Integer Set Generation (ISG) Procedure }
	\label{alg:ISG}
	\baselineskip 0.5 cm
	\small
	\begin{algorithmic}
		\vspace{.1cm}
		\STATE \textbf{Step [1] Initialize:} Let $\bar{y}$ be the lower bound of the variables in \eqref{eqn2a}, and initialized as $\bar{y}_i =\lfloor \hat{y_i}\rfloor,  \, \, \forall i =1 \ldots n_2$. Let $K' \subseteq K$ be the subset of indices for the binding constraints at current solution $\hat{y}$. Let $I$ be the set of variable indices, and $K$ be the set of constraint indices for \eqref{eqn2a}.
		\STATE \textbf{Step [2] Compute Distance $d_{ij}$:}
		\FOR {$i \in I$}
		\FOR {$j \in I\setminus i$}
		\STATE $d_{ij} = 0, f_k = 0$
		\FOR {$k \in K'$}
		\STATE (a) \textit{Assign righthand side for constraint index k:}
		\STATE $f_k \leftarrow \tau_k$
		\STATE (b) \textit{Projections for all other indices except indices i and j:}
		\FOR {$t \in I\setminus \{i,j\}$}
		\STATE $f_k \leftarrow f_k - w_{kt} \hat{y}_t$
		\ENDFOR	\\   	
		\STATE (c) \textit{Compute Distance:}
		\STATE $r_k \leftarrow f_k - w_{ki} \bar{y}_i$
		\STATE $d_{ij} \leftarrow \min \left\lbrace r_k /w_{kj},u_j\right\rbrace $
		\ENDFOR	\\   	
		\ENDFOR	\\   	
		\ENDFOR	\\   	
	\end{algorithmic}
\end{algorithm}

\begin{algorithm}[H]
	\small
	\begin{algorithmic}
		\STATE \textbf{Step [3] Evaluate Bounds:}
		\FORALL {$i \in I$}
		\FORALL {$j \in I\setminus i$}
		\FORALL {$k \in K'$}
		\STATE $z= 0$				
		\REPEAT
		\STATE $\alpha = 0$
		\STATE (d) \textit{Evaluate $d_{ij}$:}
		\IF {$(d_{ij} < 1) \, \, \& \, \, \bar{y}_i \geq 1$}
		\STATE $\bar{y}_i \leftarrow \bar{y}_i -1$; $\alpha = 1$;
		\ELSIF {$d_{ij} - \bar{y}_j < 1 \, \, \& \, \,\bar{y}_j \geq 1$}
		\STATE $\bar{y}_j \leftarrow \bar{y}_j -1$; $\alpha = 1$;
		\ENDIF
		\STATE (e) \textit{Check for Integer Points:}
		\STATE $b \leftarrow 1$
		\WHILE {$\bar{y}_i - b > 0  \, \, \& \, \, \alpha = 0 $}
		\STATE $r^{(1)}_k \leftarrow f_k - w_{ki} (\bar{y}_i-b)$; $d^{(1)}_{ij} \leftarrow r^{(1)}_k /w_{kj}$;
		\IF {$(\lfloor d^{(1)}_{ij} \rfloor - \lfloor d_{ij} \rfloor \geq 1) \, \, \& \, \, (\lfloor d^{(1)}_{ij} \rfloor \leq u_i) \, \,$}
		\STATE $\bar{y}_i \leftarrow \bar{y}_i -b$; $\alpha = 1$;
		\ENDIF \\
		$ b \leftarrow b + 1$;
		\ENDWHILE
		\IF {$\alpha = 1$}		
		\STATE Re-evaluate $d_{ij}$ using the Step [2c];
		\ENDIF   		
		\STATE $z \leftarrow z + 1$;						
		\UNTIL{$\alpha = 0$}   		
		\ENDFOR	\\   	
		\ENDFOR	\\   	
		\ENDFOR	\\   	
		\STATE \textbf{Step [4] Use the computed lower bound} $\bar{y}_i$ for the variable index $i \in I$ in FCG.
	\end{algorithmic}
\end{algorithm}

In Algorithm \ref{alg:ISG}, we initialize $\bar{y}_i =\lfloor \hat{y}_i\rfloor$ to obtain the set $F_{\upsilon}^{IP}$. In each iteration, we evaluate two components of $y$ along $i$ and $j$, and the lower bounds for the components are decreased based on the  distance of the components from the binding constraints. Each pair of components, $y_i$ and $y_j$ are evaluated in two-dimensional (2D) space. The pair of components $(y_i,y_j)$ are evaluated in the 2D space with the criterion that $\bar{y_i}$ and $\bar{y_j}$ provide at least one integer point for the generation of the valid inequality in $i^{th}$ direction. We also make sure that $\bar{y_i}$ and $\bar{y_j}$ does not remove any integer point in $C(F^{IP})$ so that the generated valid inequality does not cut off any optimal solution.

In Step [1], we initialize the parameters using LP-relaxation solution for subproblem \eqref{eqn2a} given as $\hat{y}_i$ for $i \in I$, where $I$ is the set of decision variable indices. Furthermore, $\bar{y}_i$ is the parameter of the algorithm which we intend to use as lower bound for the $i^{th}$ component in the subproblem \eqref{eqn2a}. We would like to \textit{increase} the value of $\bar{y}_i$ without cutting off any integer solution for the original problem. Initially $\bar{y}_i$ is set to $\lfloor \hat{y}_i\rfloor$. In Step [2a], the right hand side of binding constraint $k \in K'$ is assigned to parameter $f_k$. In Step [2b], for any $i,j \in I,$ such that $i \neq j$, we calculate the distance $d_{ij}$ in $(i,j)$ space using equation (\ref{dval}). We then evaluate $d_{ij}$ in  Step [2c]. The parameter $d_{ij}$ is the measure of distance from the other component's axis to the binding constraint. If $d_{ij} < 1$, then it indicates the absence of an integer point along $i^{th}$ axis, then $\bar{y_i}$ is decreased by one, and $\bar{y}$ will be evaluated again for the binding constraint $k$ using the expanded set  $F_{\upsilon}^{IP}$. Hence, we start with a smallest set of integers based on $ \lfloor \hat{y}_i\rfloor$, and as the algorithm progresses, the set $F_{\upsilon}^{IP}$ is expanded by decreasing $\bar{y}$ based on Corollary \ref{corr1} and Lemma \ref{lemma2}.

\begin{figure}[H]
	\centering
	\includegraphics[scale=1.00]{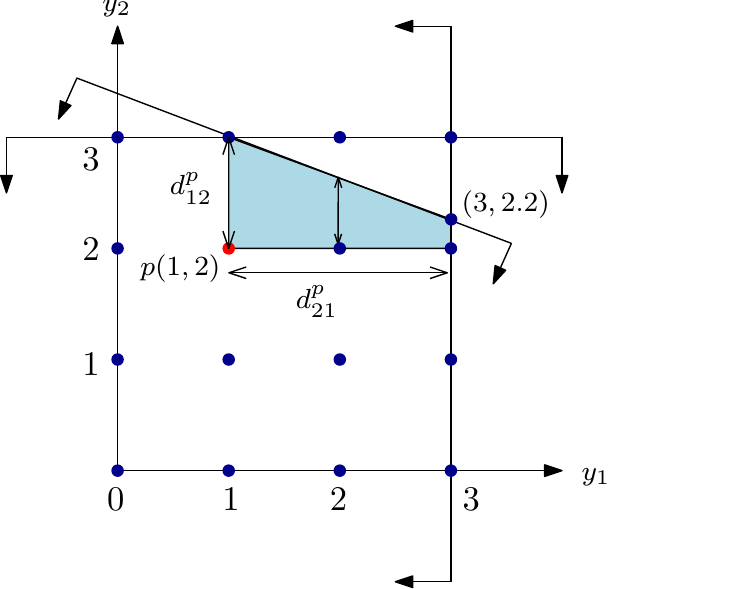}\\
	\caption{Illustration of ISG algorithm}
	\label{brp-fig1-sp}
\end{figure}

In step [3d], we check property \textit{(ii)} of Corollary \ref{corr1}. If there are any integer points along the component $i$, then $\bar{y}_i$ is reduced to accommodate additional integer points into the set $F_{\upsilon}^{IP}$. We make sure that the reduced set $C(F_{\upsilon}^{IP})$ is sufficient to get the required valid inequality. The complexity of the algorithm is $\BigO{n^3m}$, where $n$ is the number of variables, and $m$ is the number of constraints. The 2D construction from the entire polytope is similar to variable elimination method described in \cite{dantzig1973fourier}.

\subsection{Numerical Examples}
In this section, we demonstrate the ISG algorithm using numerical examples. In Example 1, we use an IP subproblem with two decision variables and one constraint to illustrate the generation of a reduced integer set using the ISG algorithm. Example 2 demonstrates the algorithm for a subproblem with two decision variables and two constraints. In both Example 1 and 2, Step[3e] of the ISG algorithm is not required, therefore we use Example 3 to demonstrate the significance of this step.\\

\noindent \textbf{Example 1:}
Consider the following IP subproblem:
\begin{equation}
\begin{alignedat}{2}\label{eq-ip1}
\text{IP1: } \Max   & y_1 + y_2 && \\
\text{s.t. }        & 0.4 y_1 + y_ 2 \leq 3.4                && \\
& 0 \leq y_1,y_2 \leq 3                 && \\
& y_1,y_2 \in \mathbb{Z},                &&
\end{alignedat}
\end{equation}

\noindent The LP-relaxation to the problem (\ref{eq-ip1}) has the optimal solution (3, 2.2). Thus, a valid inequality has to cut off the fractional solution.
The steps of Algorithm 1 are as follows:\\
\begin{itemize}
	\item[] Step [1] Initialize: $y^p =(\lfloor 3 \rfloor,\lfloor 2.2 \rfloor) = (3, 2)$. Therefore, ${y}_1^{p} = 3$, ${y}_2^{p} = 2$, $\bar{y}_1 = 3$, $\bar{y}_2 = 2,$ based on $\hat{{y}_1} = 3$ and $\hat{{y}_2} = 2.2$.
	\item[] Step [2] Compute Distance $d_{12}^p$: $i=1, j=2 , k=1$, \newline
	$r_1 = 3.4 - 0.4(3) = 2.2$. $d_{12}^p = \min \left\lbrace 2.2/1 - y_2^p, 3 - y_2^p \right\rbrace = 0.2 $.
	\item[] Step [3] Evaluate Bounds: \newline
	$z=0, i=1:$\newline
	$d_{12}^p < 1$ $\Rightarrow \bar{y}_1 = \bar{y}_1-1 = 3-1=2,$ update $y^p \Leftarrow (2,2)$, and $\alpha= 1$. \newline
	Since $\alpha= 1$, Re-evaluate $d_{12}^p$: $r_1 = 3.4 - 0.4(2) = 2.6$. $d_{12}^p = \min \left\lbrace 2.6/1 - y_2^p, 3 - y_2^p \right\rbrace = 0.6 $. \newline
	$z=1,i=1:$\newline	
	$d_{12}^p < 1$ $\Rightarrow \bar{y}_1 = \bar{y}_1-1 = 2-1=1,$ update $y^p \Leftarrow (1,2)$, and $\alpha= 1$. \newline	
	Since $\alpha= 1$, Re-evaluate $d_{12}^p$: $r_1 = 3.4 - 0.4(1) = 3.0$. $d_{12}^p = \min \left\lbrace 3.0/1 - y_2^p, 3 - y_2^p \right\rbrace = 1.0 $. \newline
	$z=2,i=1:$\newline
	$d_{12}^p \geq 1$, then evaluate the next component $\bar{y}_2$.\newline	 
	$\bar{y}_2 = 2$, $r_1 = 3.4 - 2 = 1.4$. \newline
	$d_{21}^p = \min \left\lbrace 1.4/0.4 - y_1^p, 3 - y_1^p \right\rbrace = \min \left\lbrace 3.5 - y_1^p, 3 - y_1^p \right\rbrace = 2$. \newline 	
	$z=0,i=2:$\newline		
	Since $d_{21}^p \geq 1$, we do not make any changes to $\bar{y}_2 $.
	\item[] Step [4] New Lower Bounds: $y_1=1$ and $y_2=2$.				
\end{itemize}

The value $d_{12}^p=0.2$ represents the distance between the point $y^p(3, 2)$ and the point $y^p(3, 2.2)$ for a binding constraint in the $y_2$ direction. Since there is no integer point in the direction, the algorithm iterates to reach the point $y^{p'}(2, 2)$, where the distance $d_{12}^{p'} = \min \left\lbrace 2.6/1 - y_2^{p'}, 3 - y_2^{p'} \right\rbrace = 0.6$. The value `$0.6$' is the distance between the point $y^{p'}(2,2)$ and the binding constraint in $y_2^{p'}$ direction. Since $d_{12}^{p'} < 1$, the algorithm is continued to next iteration. The value $d_{12}^{p''} = \min \left\lbrace 3/1 - y_2^{p''}, 3 - y_2^{p''} \right\rbrace = 1$ is the distance between the point $y^{p''}(1,2)$ and the binding constraint at the point $y^{p''}(1,2)$.

The feasible set based on the new point $(1,2)$ as the origin is depicted in Figure \ref{brp-fig1}-(a). The reduced feasible set is now used for the generation of a cut, which is depicted in Figure \ref{brp-fig1}-(b).

\begin{figure}[H]
	\centering
	\includegraphics{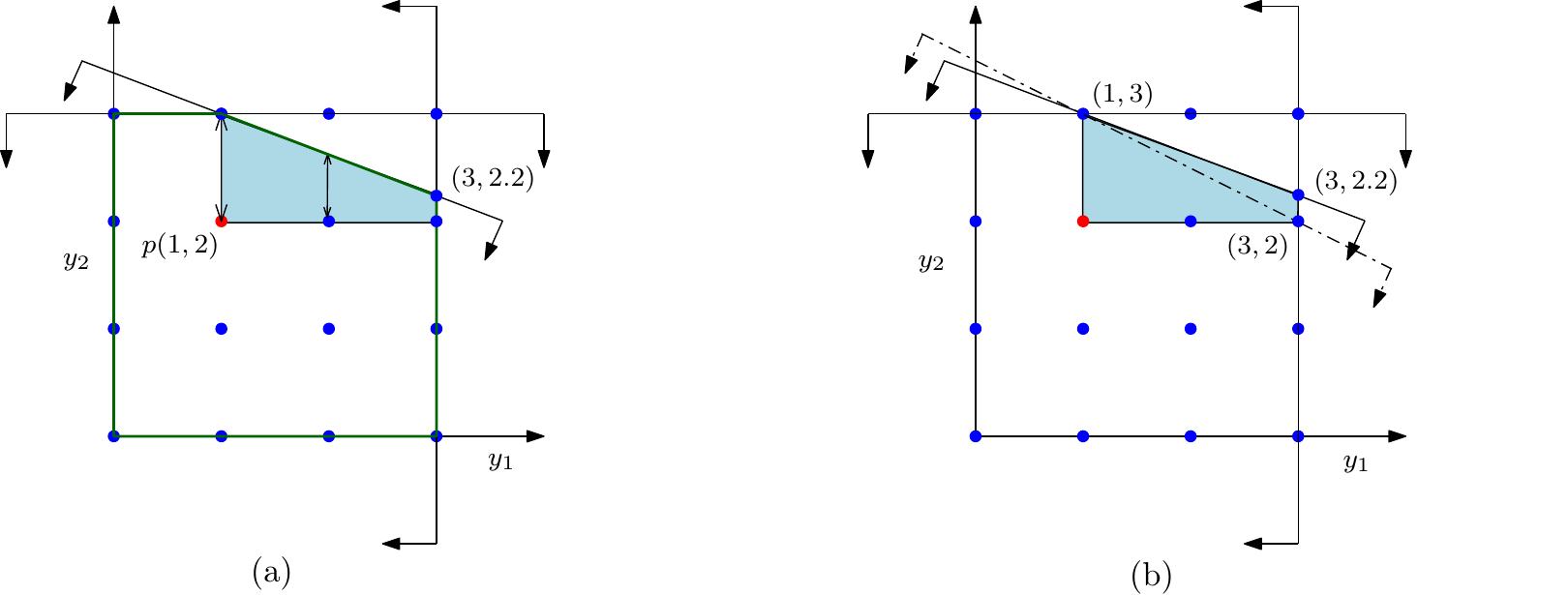}\\
	\caption{Example 1 illustration: (a) reduced integer set and (b) cut generated based on the reduced set}
	\label{brp-fig1}
\end{figure}

\noindent \textbf{Example 2:}
Consider the following IP subproblem:
\begin{subequations}\label{eq-ip2}
\begin{alignat}{3}
\text{IP2: } \Max   & y_1 + y_2 && \\ \label{eq-ip2a1}
\text{s.t. }        & 0.4 y_1 + y_ 2 \leq 3.4                && \\ \label{eq-ip2a2}
			        &  y_1 + 0.4 y_ 2 \leq 3.4                && \\ \label{eq-ip2a3}
& 0 \leq y_1,y_2 \leq 3                 && \\ \label{eq-ip2a4}
& y_1,y_2 \in \mathbb{Z},                &&
\end{alignat}
\end{subequations}

\noindent The LP-relaxation to the problem (\ref{eq-ip2}) has the optimal solution (2.42, 2.42). Thus, a valid inequality has to cut off the fractional solution.
The steps of Algorithm 1 are as follows:\\
\begin{itemize}
	\item[] Step [1] Initialize: $y^p =(\lfloor 2.42 \rfloor,\lfloor 2.42 \rfloor) = (2, 2)$. Therefore, ${y}_1^{p} = 2$, ${y}_2^{p} = 2$, $\bar{y}_1 = 2$, $\bar{y}_2 = 2,$ based on $\hat{y}_1 = 2.42$ and $\hat{{y}_2} = 2.42$.
	\item[] Step [2] Compute Distance : \newline
	$i=1, j=2 , k=1$, \newline
	$r_1 = 3.4 - 0.4(2) = 2.6$. $d_{12}^p = \min \left\lbrace 2.6/1 - y_2^p, 3 - y_2^p \right\rbrace = 0.6 $. \newline
	\item[] Step [3] Evaluate Bounds : \newline
	$z=0, i=1, k=1:$\newline	
	$d_{12}^p < 1$ $\Rightarrow \bar{y}_1 = \bar{y}_1-1 = 2-1=1,$ update $y^p \Leftarrow (1,2)$, and $\alpha= 1$. \newline
	Since $\alpha= 1$, Re-evaluate $d_{12}^p$: $r_1 = 3.4 - 0.4(1) = 3.0$. $d_{12}^p = \min \left\lbrace 3.0/1 - y_2^p, 3 - y_2^p \right\rbrace = 1.0 $. \newline
	$z=1, i=1, k=1:$\newline		
	$d_{12}^p \geq 1$, then evaluate the next constraint.\newline	
	$i=1, j=2 , k=2$, \newline
	$r_2 = 3.4 - 1 = 2.4$. $d_{12}^p = \min \left\lbrace 2.4/0.4 - y_2^p, 3 - y_2^p \right\rbrace = 1.0 $. \newline
	$z=0, i=1, k=2:$\newline		
	$d_{12}^p \geq 1$, and all the constraints are evaluated, so we move to the next component $\bar{y}_2$.\newline	
	$z=0,i=2, j=1 , k=1$, \newline		
	Similarly, for $\bar{y}_2 = 2$, $r_1 = 3.4 - 2 = 1.4$. \newline
	$d_{21}^p = \min \left\lbrace 1.4/0.4 - y_1^p, 3 - y_1^p \right\rbrace = \min \left\lbrace 3.5 - y_1^p, 3 - y_1^p \right\rbrace = 2$. \newline 	
	$d_{21}^p \geq 1$, then evaluate the next constraint.\newline	
	$z=0, i=2, j=1 , k=2$, \newline
	$r_2 = 3.4 - 0.4(2) = 2.6$.\newline	
	$d_{21}^p = \min \left\lbrace 2.6/1.0 - y_1^p, 3 - y_1^p \right\rbrace = \min \left\lbrace 2.6 - y_1^p, 3 - y_1^p \right\rbrace = 1.6$. \newline 	
	Since $d_{21}^p \geq 1$, we don't make any changes to $y_2$.
	\item[] Step [4] New Lower Bounds: $y_1=1$ and $y_2=2$.				
\end{itemize}

The value $d_{12}^p=0.6$ represents the distance between the point $y^p(2, 2)$ and the point $y^p(2, 2.6)$ for the binding constraint \eqref{eq-ip2a1} in $y_2$ component's direction. Since there is no integer point in the direction, the algorithm iterates to reach the point $y^{p'}(1, 2)$, where the distance $d_{12}^{p'} = \min \left\lbrace 3.0/1 - y_2^{p'}, 3 - y_2^{p'} \right\rbrace = 1.0$. The value `$1.0$' is the distance between the point $y^{p'}(1,2)$ and the binding constraint \eqref{eq-ip2a1} in $y_2^{p'}$ direction. Since $d_{12}^{p'} \geq 1$, the algorithm considers the next constraint. Similarly, the other index $y_2$ is evaluated in $y_1$'s directions for both the constraints \eqref{eq-ip2a1} and \eqref{eq-ip2a2}.

The feasible set based on the new origin $(1, 2)$ is depicted in Figure \ref{brp-fig2}-(a). The feasible set is used for the generation of valid inequalities, and the generated cut is shown in Figure \ref{brp-fig2}-(b). After performing to the ISG algorithm, the new origin is (1, 2).

\begin{figure}[H]
	\centering
	\includegraphics{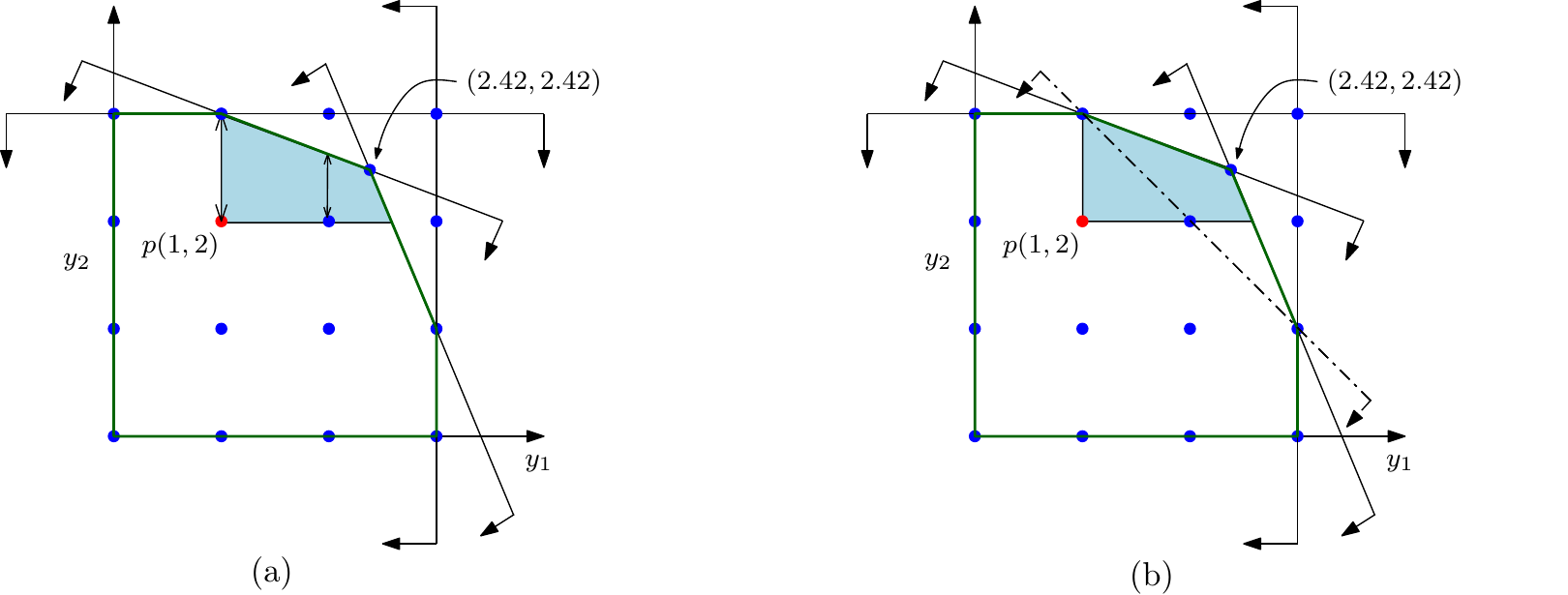}\\
	\caption{Example 2 illustration: (a) reduced integer set and (b) cut generated based on the reduced set}
	\label{brp-fig2}
\end{figure}

\noindent \textbf{Example 3:}
Consider another IP subproblem given as follows::
\begin{equation}
\begin{alignedat}{2}\label{eq-ip3}
\text{IP3: } \Max   & 1.2 y_1 + 3.4 y_2 && \\
\text{s.t. }        & 6 y_1 + 5 y_ 2 \leq 37.4                && \\
& 0 \leq y_1,y_2 \leq 5                 && \\
& y_1,y_2 \in \mathbb{Z},                &&
\end{alignedat}
\end{equation}

The LP-relaxation for problem (\ref{eq-ip3}) gives the solution $(5,1.48)$. Using Step [3d] of the ISG algorithm, the new origin for FCG procedure is shifted to $(4,0)$. However, based on  $F_{\upsilon}^{IP}$, the generated valid inequality removes an integer point $(2,5)$ from the solution space. Hence, we use Step [3e] to prevent any possibility of removing off integer points from the solution space based on the current reference obtained from step [3d]. Thus Step [3e] gives the new reference $(2,0)$ based on the possible integer points in $F^{IP}$. The reduced solution space based on the new origin $(4,0)$ without Step (e) is shown in Figure \ref{brp-fig3}-(a). The reduced solution space based on Step [e] for the generation of a valid inequalities is shown in Figure \ref{brp-fig3}-(b).

\begin{figure}[H]
	\centering
	\includegraphics[scale=0.80]{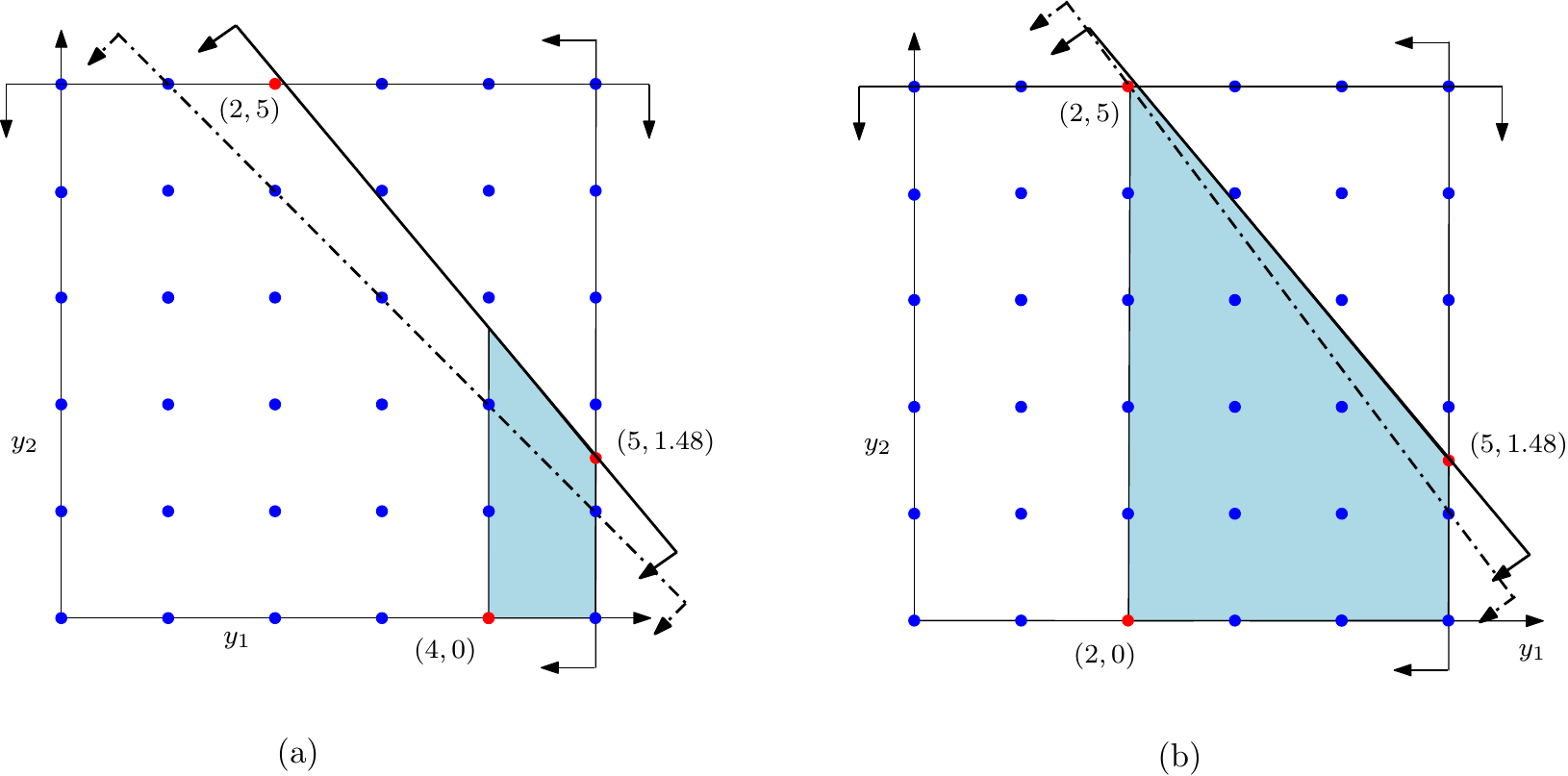}\\
	\caption{Example 3 illustration: (a) reduced integer set and (b) cut generated based on the reduced set}
	\label{brp-fig3}
\end{figure}

The newly computed lower bounds from ISG algorithm will be used for the generation of valid inequalities. The objective of ISG algorithm is to obtain a smaller reduced set $F_{\upsilon}^{IP} \subseteq F^{IP}$ which further gives $C(F_{\upsilon}^{IP}) \subseteq C(F^{IP})$. A reduced set $F_{\upsilon}^{IP}$ is expected to provide better runtime for the generation of a valid inequality.

\section{Fenchel Decomposition Algorithm}\label{sec-alg}
To demonstrate how to use the ISG method within an SMIP algorithm, we will apply this method to stage-wise Fenchel decomposition (SFD) for solving SMIPs \cite{Ntaimo10}. SFD adopts the Benders' decomposition setting with $x$ as the first-stage decision variable in the master problem, and $y(\omega)$ as the second-stage decision variable in the subproblem. In SIP2, instead of working with the IP subproblem directly, SFD seeks to find the optimal solution via a cutting plane approach on a partial LP-relaxation of SIP2 where only the subproblems are relaxed. Fenchel decomposition (FD) cuts are sequentially generated to recover (at least partially) the convex hull of integer points for each scenario subproblem feasible set. If a subproblem LP has a non-integer solution, an FD is generated and added to cut off the fractional solution. FD cuts are capable of recovering faces of the convex hull of integer programs, which is the special structure for SIP2. The goal is to construct the convex hull of integer points in the neighborhood of the optimal solution so that by solving subproblem LPs with sufficient FD cuts added, we can find the optimal solution without having to resort to a branch-and-bound scheme to guarantee optimality.

\subsection{Algorithm}\label{sec-alg}

At a given iteration $k$ of the SFD cutting plane algorithm, the master problem takes the following form: \\
\begin{subequations}\label{eq-master-1}
	\begin{align}
	z^k = \Max      \   & c^\top x + \theta          \notag \\
	\text{s.t. } & Ax \leq b  \notag\\
	& (\eta^{t})^\top x + \theta \geq \gamma^{t} , \ t \in 1,...,k \label{eq-master-1a} \\
	& x \in \{0,1\}.  \notag
	\end{align}
\end{subequations}
In the master problem \eqref{eq-master-1}, $\theta$ is the optimality cut decision variable, $\eta \in \Re^{n1}$ is the optimality cut coefficient vector, and $\gamma \in \Re$ is the right hand side. Constraints \eqref{eq-master-1a} are the \textit{optimality} cuts, which are computed based on the optimal dual solutions of all the LP-relaxation subproblems. Optimality cuts approximate the value function of the second-stage subproblems. For a first-stage solution $x^k$ from the master problem \eqref{eq-master-1}, the subproblem for each scenario $\om \in \Omega$ is given as follows:
\begin{subequations}\label{eq-suSIP2}
	\begin{align}
	\text{SP}(\om):\Phi^k_{LP}(\rho,\tau,\om) = \Max \ & \rho^{\top} y(\om)            \notag\\
	\text{s.t. } & W y(\om) \leq \tau \notag \\
	& \beta^t(\om)^{\top} y(\om) \leq g(\om,\beta^t(\om)), \ t \in \Theta(\om) \label{eq-suSIP2a}\\
	& 0 \leq y(\om) \leq u	\notag \\               	
	& y(\om) \geq 0.               \notag
	\end{align}
\end{subequations}
Constraints \eqref{eq-suSIP2a} are the Fenchel cuts, and $\Theta(\om)$ is the index set for algorithm iterations at which a Fenchel cut is generated for each $\om \in \Omega$. Next, we describe how these cuts are generated.

The SFD algorithm starts by initializing data in Step [1] and getting an initial solution by solving the LP-relaxation of SIP2 in Step [2]. If the initial solution satisfies the integrality restrictions for all subproblems in Step [3], i.e., $x \in X$ and $y(\omega) \in Y, \, \forall \, \om \in \Omega$, then the solution is declared $\epsilon$-optimal, and the algorithm terminates. Otherwise, the algorithm continues by calculating and storing the optimality cut coefficients for all subproblems with an integer solution in Step [4].

\begin{algorithm}[H]
	\caption{Stage-Wise Fenchel Decomposition (SFD) Algorithm}
	\label{alg:FCG}
	\baselineskip 0.5 cm
	\small
	\begin{algorithmic}
		\vspace{.1cm}
		\STATE \textbf{Step [1] Initialization:} set $k \leftarrow 0, \epsilon > 0, LB \leftarrow -\infty$ and $UB \leftarrow \infty$.
		\STATE \textbf{Step [2] Get initial solution:} Solve problem (\ref{eq-master-1}-\ref{eq-suSIP2}) using the L-shaped algorithm to get  solution $(\hat{x}^0,\hat{y}^0(\omega))$, objective function value $\varphi^0 = \sum_{\omega \in \Omega}p_{\omega} \Phi^k_{LP}(\rho,\tau,\om)$, and dual solutions $\hat{\pi}^k(\omega)$ for each $\omega \in \Omega$.\\
		\STATE \textbf{Step [3] Check solution integrality:} \\
		\IF {$\hat{y}^k(\om)$ $\in Y$}
		\STATE Report $(\hat{x}^k,\hat{y}^k(\omega))$ as optimal.\\
		\STATE \textbf{Stop}.
		\ENDIF
		\STATE \textbf{Step [4] Calculate and store optimality cuts coefficients for scenarios with integer solution}
		\FOR {$\om \in \Omega$}
		\IF {$\hat{y}(\omega)^k$ $\in Y$}
		\STATE Calculate and store optimality cut coefficients $\eta(\omega)^k \leftarrow \hat{\pi}(\omega)^{k\top} T(\omega)$ and $\gamma(\omega)^k \leftarrow \hat{\pi}(\omega)^{k\top} h(\omega)$.
		\ENDIF
		\ENDFOR
		\STATE \textbf{Step [5] Fenchel cuts and optimality cuts generation:} \\
		\FOR {$\om \in \Omega$}
		\IF {$\hat{y}(\omega)^k \notin Y$}
		\STATE \textit{Compute scenario Fenchel cut coefficients:} Run ISG to get $F_{\upsilon}^{IP}$ and use FCG  based on $F_{\upsilon}^{IP}$ to get $\beta(\omega)^k$ and $g(\omega,\beta(\omega)^k)$.
		\STATE Add the cut $\beta(\omega)^{k\top} y(\omega) \leq g(\omega,\beta(\omega)^k)$ to subproblem \eqref{eq-suSIP2}.
		\STATE Solve the updated subproblem $\text{SP}(\omega)$ and get updated subproblem dual solution $\hat{\pi}(\omega)^k$.
		\STATE Update optimality cut coefficients $\eta^k \leftarrow \eta^k + p_{\omega} \cdot (\hat{\pi}(\omega)^{k})^\top T(\omega)$ and $\gamma^k \leftarrow \gamma^k + p_{\omega} \cdot (\hat{\pi}(\omega)^{k})^\top h(\omega)$.
		\ENDIF
		\ENDFOR
	\end{algorithmic}
\end{algorithm}

\begin{algorithm}[H]
	\small
	\begin{algorithmic}
		\STATE \textbf{Step [6] Add optimality cut} $\eta^k x + \theta \geq \gamma^k$ to master problem \eqref{eq-master-1} and update iterator: set $k \leftarrow k+1$.
		\STATE \textbf{Step [7] Solve master problem} \eqref{eq-master-1} to get a new first-stage solution $\hat{x}^k$ and objective value $z^k$.
		\STATE \textbf{Step [8] Update lower bound:} Set $LB \leftarrow max\{LB, z^k\}$
		\STATE \textbf{Step [9] $\epsilon$-optimality check:}
		\IF {$|UB-LB| \leq \epsilon |LB|$}
		\STATE Go to Step [14].
		\ENDIF
		\STATE \textbf{Step [10] Solve subproblems:}
		\FOR {$\omega \in \Omega$}
		\STATE Solve subproblem \eqref{eq-suSIP2} to get updated subproblem solution $\hat{y}(\omega)^k$, optimal value $\Phi^k_{LP}(\rho,\tau,\om)$ and dual solution $\hat{\pi}(\omega)^k$.
		\IF {$\hat{y}(\omega)^k$ $\in Y$}
		\STATE Calculate and store optimality cut coefficients $\eta(\omega)^k \leftarrow \hat{\pi}(\omega)^{k\top} T(\omega)$ and $\gamma(\omega)^k \leftarrow \hat{\pi}(\omega)^{k\top} h(\omega)$.
		\ENDIF
		\ENDFOR
		\STATE \textbf{Step [11] Subproblem solutions integrality check:}
		\FOR {$\omega \in \Omega$}
		\IF {$y(\omega)^k \notin Y$}
		\STATE Go to Step [5].
		\ENDIF
		\ENDFOR

		\STATE \textbf{Step [12] Update solution and bound information:}
		\STATE $\,\,\,\,$ Update incumbent solution: $x^* \leftarrow x^k$.
		\STATE $\,\,\,\,$ Update upper bound: $UB \leftarrow min \{UB, c^{\top}{x}^k + \sum_{\omega \in \Omega} p_{\omega} \Phi^k_{LP}(\rho,\tau,\om)  \}$.
		\STATE \textbf{Step [13] $\epsilon$-optimality check:}
		\IF {$|UB-LB| > \epsilon |LB|$}
		\STATE Go to Step [6].
		\ENDIF
		\STATE \textbf{Step [14] Declare} $x^*$ $\epsilon$-optimal.
		\STATE \textbf{Stop.}
	\end{algorithmic}
\end{algorithm}

For subproblems with a solution that does not satisfy the integrality requirements, Fenchel cut coefficients $\beta^k(\omega)$, and the right hand side $g(\omega, \beta^k(\om))$ are computed for the iteration $k$ in Step [5]. A Fenchel cut is added to subproblem $\text{SP}(\omega)$. Next, the dual solution obtained by solving the subproblem is used to generate the optimality cut coefficients. Once all subproblems have been solved at a given iteration, the optimality cut is added to the master problem  in Step [6]. The iteration counter $k$ is incremented by one, and the master problem is solved again in Step [7] to get an updated first-stage solution and objective value.

The lower bound $LB$ is updated in Step [8]. The gap between the lower bound $LB$ and the upper bound $UB$ is verified in Step [9]. If this gap is small enough, then the incumbent solution is declared $\epsilon$-optimal in Step [14], and then the algorithm terminates. Otherwise, all the subproblems are solved again, and optimality cut coefficients are updated for subproblems with an integer solution in Step [10]. The integrality of subproblem solutions is verified in Step [11]: if a subproblem solution for any given $\omega$ is not integral, the algorithm returns to Step [5], to generate and add Fenchel cuts to the subproblems with a non-integer solution, and compute their optimality cut coefficients. Otherwise, the incumbent solution $x^*$ and the upper bound $UB$ are updated in Step [12]. The optimality check is done again in Step [13]: if it is satisfied, the incumbent solution is $\epsilon$-optimal, and the algorithm is terminated. Otherwise, the algorithm returns to Step [6], and the optimality cut is added to the master problem, and its solved again. The algorithm is continued until the termination condition is satisfied.

\subsection{Fenchel Cut Generation with Reduced Integer Set}
For the sake of completeness, we next present the Fenchel cut generation procedure which is based on \cite{ntaimo2013fenchel}. The procedure uses a master problem to construct a linear approximation of the subproblem space while the subproblem returns feasible integer points from $F_{\upsilon}^{IP}$. The master problem is given as,

\begin{equation}\label{eq-cutlpmaster1}
\begin{alignedat}{2}
\delta^{(t)} = && \ \ \Max_{\beta(\omega) \in \Pi^{\beta}} &  \ \   \theta  \\
&& \text{s.t. }  \ \ \ &  -\theta + (\hat{y(\omega)} - y(\omega)^{(\nu)})^\top \beta(\omega)^{(\nu)} \geq 0, \ \nu=1,\cdots,t.
\end{alignedat}
\end{equation}

\noindent where the subproblem at iteration $t$ is solved to get $y(\omega)^{(\nu)}$.
\begin{equation}\label{gee1}
g(\omega,\beta(\omega)) =
\Max_{y(\omega) \in F_{\upsilon}^{IP}}\left\{ \beta(\omega)^{\top} y(\omega) \right\}.
\end{equation}

Solving the problems \eqref{eq-cutlpmaster1} and \eqref{gee1} iteratively will give the optimal $\beta(\omega)$ and $g(\omega,\beta(\omega))$. The algorithm is stated as follows:

\begin{algorithm}[H]
	\caption{Fenchel Cut Generation Procedure (FCG)}
	\label{alg:FCG1}
	\baselineskip 0.5 cm
	\small
	\begin{algorithmic}
		\vspace{.1cm}
		\STATE \textbf{Step [1] Initialization:} Set $t \leftarrow 0, \epsilon > 0, LB \leftarrow -\infty, UB \leftarrow \infty$, and get an initial point $\beta(\omega)^{(0)} \in \Pi^{\beta}$.
		\STATE \textbf{Step [2] Lower Bound:}
		\STATE $\,\,\,\,$ Use $\beta(\omega)^{(t)}$ to solve problem (\ref{gee1}) and get solution $y(\omega)^{(t)}$ and the corresponding objective value $g(\omega,\beta(\omega)^{(t)})$.
		\STATE \textbf{Compute lower bound:}
		\STATE $\,\,\,\,$ Let $d^{(t)} \leftarrow (\hat{y(\omega)}-y(\omega)^{(t)})^\top\beta(\omega)^{(t)}$.
		\STATE $\,\,\,\,$ Set $l^{(t+1)} \leftarrow max\{d^{(t)},l^{(t)}\}$.
		\IF {$l^{(t+1)}$ is updated}
		\STATE \textbf{Update incumbent solution:}
		\STATE $\,\,\,\,$ Set $\mu \leftarrow d^{(t)}$ and $(\beta(\omega)^{*},g(\omega,\beta(\omega)^{*})) \leftarrow (\beta(\omega)^{(t)},g(\omega,\beta(\omega)^{(t)}))$.
		\ENDIF
		\STATE Use $\hat{y(\omega)}$ and solution ${y}(\omega)^{(t)}$ from subproblem (\ref{gee1}) to form and add constraint to the problem (\ref{eq-cutlpmaster1}).
		\STATE \textbf{Step [3] Upper Bound:}
		\STATE $\,\,\,\,$ Solve problem (\ref{eq-cutlpmaster1}) to get an optimal solution
		 $(\theta^{(t)},\beta(\omega)^{(t)})$.
		\STATE \textbf{Compute upper bound:}
		\STATE $\,\,\,\,$ Set $u^{(t+1)} \leftarrow min\{\theta^{(t)}, u^{(t)}\}$.
		\IF {$u^{(t+1)} - l^{(t+1)} \leq \epsilon'$}
		\STATE The incumbent solution is optimal.
		\STATE \textbf{Stop.}
		\ELSE
		\STATE Set $t \leftarrow t+1$ and go to [2].
		\ENDIF
	\end{algorithmic}
\end{algorithm}

In Step [1] we initialize the parameters for the algorithm. Each component of $\beta(\omega)^{(0)}$ is arbitrarily initialzed to a value with their bounds. Since a subproblem integer program \eqref{eqn2a} has to be solved many times to generate Fenchel cuts, a linearly constrained domain for $\Pi^{\beta}$ such as the $L^1$ unit sphere or $L^2$ can also be used. Step [2] uses $\beta(\omega)^{(0)}$ as coefficients, then subproblem (\ref{gee1}) is solved, and the corresponding objective value is stored. Due to ISG, subproblem (\ref{gee1}) can be evaluated using $F_{\upsilon}^{IP}$ instead of $F^{IP}$. It should be noted that subproblem (\ref{gee1}) is solved as an IP, so the solution $y(\omega)^{(t)}$ is integral. The bounds and incumbent solutions are updated in Step [2]. Based on the solution $y(\omega)^{(t)}$ from subproblem (\ref{gee1}), the cut is added to master problem (\ref{eq-cutlpmaster1}). In Step [3], master problem (\ref{eq-cutlpmaster1}) is solved and the termination condition is checked. Based on the termination condition, the algorithm either stops or continues.

\section{Computational Study} \label{sec-lcomptns}
 To gain insights into the benefits of imbedding integer set reduction into a cutting plane algorithm, we performed a computational study based on an implementation of the SFD algorithm with the option of turning on ISG. The algorithm was implemented in C++ using the CPLEX 12.5 Callable Library \cite{CPLEX} in Microsoft Visual Studio 2010. Computations were performed on an ACPI x64 computer with Intel\textregistered Xeon\textregistered Processor E5620 (2.4 GHz) and 12GB RAM. CPLEX MIP and LP solvers were used to optimize the master program and subproblems. Two test sets were created based on the two-stage multidimensional knapsack problem. The SFD algorithm was run to solve the test instances to optimality or stopped when a CPU time limit of 3,600 seconds (s) or 7,200s was reached. As a benchmark, the CPLEX MIP solver was applied to the deterministic equivalent problem (DEP) of each test instance. Next we describe how the multidimensional knapsack test instances were generated and then report the computational results in Section \ref{sec-results}.

\subsection{Test Instances Generation}\label{sec-testgen}
We created a test problem with knapsack constraints in the first-stage, and both knapsack and assignment constraints in the second-stage. The first-stage problem is given as follows:
\begin{equation}
	\begin{alignedat}{2}\label{eq-smip1}
		\Max   & \sum_{i=1}^{n_1} c_i^\top x_i +  \mathcal{Q}_{E}(x) && \\
		\text{s.t. }        & \sum_{i=1}^{n_1} x_i \leq b           && \\
		& x_i \in \{0,1\}, \, \, \forall i=1\ldots n_1.                               &&
	\end{alignedat}
\end{equation}

In problem \eqref{eq-smip1}, $c \in \Re^{n_1}$ is the first-stage cost vector, and $b \in \Re$ is the first-stage right hand side. The function $\mathcal{Q}_{E}(x)$ is the expected recourse function  given as
\begin{equation}\label{eqn2}
	\mathcal{Q}_{E}(x) = \mathbb{E}_\omega \Phi(q(\om),h(\om),T(\om)x),
\end{equation}
where $\om$ is a multivariate random variable and $\mathcal{Q}_{E}(.)$ denotes the mathematical expectation operator satisfying $\mathbb{E}_\omega \left[ \mid\Phi(q(\om),h(\om),T(\om)x)\mid\right]  < \infty$. The underlying probability distribution of $\om$ is discrete with a finite number of realizations (scenarios) with sample space $\Omega$, and corresponding probabilities $p_\om, \omega \in \Omega$. Thus for a given scenario $\om \in \Omega$, the recourse function $\Phi(q(\om),h(\om),T(\om)x)$ is given by the following second-stage MIP:
\begin{equation}
	\begin{alignedat}{2}\label{eq-SIP2}
		\Phi(q(\om),h(\om),T(\om)x) = \Max         & \sum_{i=1}^{n_2} q(\om)^{i\top} y(\om)^{i}       && \\
		\text{s.t. } & \sum_{i=1}^{n_2} w^{ij} y(\om)^{i} \leq  \sum_{i=1}^{n_1} m \cdot t^i x_i \qquad \forall j=1\ldots m_1 && \\
		& \sum_{i=1}^{n_2} v^{ik} y(\om)^{i} \leq h(\om)^k \qquad \forall k=1\ldots m_2 && \\
		& \sum_{i=1}^{n_2} u^{i\ell} y(\om)^{i} = r^\ell \qquad \forall \ell=1\ldots m_3 && \\
		& 0 \leq y(\om)^i \leq u^i, y(\om)^i \in \mathbb{Z^+} \qquad \forall i=1\ldots n_2.               &&
	\end{alignedat}
\end{equation}
In formulation \eqref{eq-SIP2}, $y(\om)$ is the recourse decision vector,  $q(\om) \in \Re^{n_2}$ is the recourse cost vector, $w^{ij} \in \Re^{+}$ is a fixed recourse parameter, and $t^{i} \in \Re^{+}, v^{ik} \in \Re^{+}, u^{i\ell} \in \Re^{+}$ is a parameter taking values 0 or 1, $m$ is a constant, and $h(\om)^{k} \in \Re^{m_2}, r^{\ell} \in \Re^{m_3}$ are the right hand side parameters. The decision vector $y(\om)$ is bounded above by vector $u$. Finally, $\mathbb{Z^+}$ is the set of nonnegative integers. Observe that formulation \eqref{eq-smip1}-\eqref{eq-SIP2} has knapsack constraints in both the first- and second-stages.

In a supply chain context, the first-stage decision vector $x$ specifies the selection of facilities, mode of transportation, and/or resources. For a realization $\om$, the second-stage decision vector $y(\om)$ could be the amount of products produced or transported based on the strategic decision $x$ from the first-stage. Additionally, knapsack-type constraints are added to represent capacity limitations in the second-stage.

Test instance data were randomly generated using the uniform distribution ($\mathcal{U}$) with different parameter values. The knapsack weights were generated by sampling from $\mathcal{U}(2, 8)$. Objective function coefficients were generated with the first-stage costs chosen so that they are much larger than second-stage costs. The first- and second-stage objective function coefficients were generated by sampling from $\mathcal{U}(0, 1500)$ and $\mathcal{U}(10, 20)$, respectively. To generate tighter knapsack constraints, the right hand side value of each constraint was generated by finding the maximum knapsack weight ($W_{max}$) for the constraint, and then sampling from $\mathcal{U}(2+(2W_{max}*v_{ub}), 4W_{max}*v_{ub})$, where $v_{ub}$ is the upper bound for the integer variables. We assume that each scenario has equal probability of occurrence.

The problem characteristics are given in Table \ref{tab:kpdimension}. The columns of the table are problem name, `Scens' is the number of scenarios, `Bvars' is the number of binary variables, `Constr' is the number of constraints, and `Nzeros' is the number of non-zero elements for each of the problem instances. The problem name has the form $k.m.n.S$, where $k$ stands for `knapsack', $m$ and $n$ is the number of first- and second-stage decision variables, respectively, and $S$ is the number of scenarios. Two test sets, `Set 1' and 'Set 2' were created, where the first set has relatively smaller size instances compared to the second set. Specifically, Set 1 has test instances with 10 binary and 20 general integer variables in the first- and second-stage, respectively. In this set the test instances were created for 50, 100, 150, and 200 scenarios. Set 2 has test instances with 20 binary and 30 general integer variables in the first- and second-stage, respectively. In the instances, first-stage has 10 constraints and second-stage has 20 and 30 constraints in Set 1 and Set 2, respectively. Test instances for this set were created for 500, 1,000, 1,500, and 2,000 scenarios. Five randomly generated replications were created for each instance size to avoid pathological cases.

\begin{table}[!ht]
	\begin{center}
		\begin{tabular}{|l|c|c|c|c|c|}
			\hline
			Problem    & Scens & Bvars & Ivars & Constr & Nzeros \\
			\hline
			\multicolumn{6}{|c|}{\emph{Set 1 }} \\
			\hline
			k.10.20.50      & 50 & 10 & 1,000 & 1,010 & 12,510\\
			k.10.20.100     & 100 & 10 & 2,000 & 2,010 & 25,010\\
			k.10.20.150     & 150 & 10 & 3,000 & 3,010 & 37,510\\
			k.10.20.200     & 200 & 10 & 4,000 & 4,010 & 50,010\\
			\hline
			\multicolumn{6}{|c|}{\emph{Set 2}} \\
			\hline
			k.10.30.50      & 50 & 10 & 1,500 & 1,510 & 16,510\\
			k.10.30.100     & 100 & 10 & 3,000 & 3,010 & 33,010\\
			k.10.30.150     & 150 & 10 & 4,500 & 4,510 & 49,510\\
			k.10.30.200     & 200 & 10 & 6,000 & 6,010 & 66,010\\
			\hline
		\end{tabular}
	\end{center}
	\caption{Test instance characteristics in terms of the DEP }
	\label{tab:kpdimension}
\end{table}

\subsection{Results} \label{sec-results}
Detailed computational results for Set 1 and Set 2 are reported in Tables \ref{tab:t1} and \ref{tab:t2}, respectively. In the tables, `SFD' and `SFD-R' represent the results using SFD algorithm without ISG and with ISG algorithm, respectively. In the tables the column `Instance' is the name of the test instance. The three columns under SFD and SFD-R, respectively,  are as follows: `MIPs' is the number of MIPs solved using the respective procedure; `FCuts' is the number of Fenchel cuts; and `\%Gap' is the percentage gap between the lower bound (LB) and the upper bound (UB) value after the stipulated runtime (3600s for Set 1 and 7200s for Set 2). Finally,  the last column `\%Gap' shows the CPLEX MIP gap after solving the DEP for the designated amount of time.

\begin{table}[htpb]
	\centering
	\scalebox{1.00}{		
	\resizebox{\columnwidth}{!}{%
			\begin{tabular}{|c|c|c|c|c|c|c|c|c|}
			\toprule
			& \multicolumn{1}{c|}{} & \multicolumn{3}{c|}{SFD} & \multicolumn{3}{c|}{SFD-R} & \multicolumn{1}{c|}{CPLEX} \\
			\hline
			No. & Instance & MIPs & FCuts & \%Gap & MIPs & FCuts & \%Gap  & \%Gap \\
			\hline 
			1 & k.10.20.50a & 38,523 & 376 & 5.36 & 60,049 & 564 & 3.58 & 10.05\\
			2 & k.10.20.50b & 42,435 & 392 & 4.23 & 59,263 & 539 & 2.86 & 8.83\\
			3 & k.10.20.50c & 40,203 & 384 & 5.05 & 62,035 & 576 & 3.29 & 10.19\\
			4 & k.10.20.50d & 37,262 & 360 & 5.43 & 58,391 & 540 & 3.83 & 9.67\\
			5 & k.10.20.50e & 40,254 & 384 & 4.53 & 61,607 & 576 & 2.90 & 8.45\\\hline
			& Average& 39,735 & 379 & 4.92 & 60,269 & 559 & 3.29 & 9.44\\\hline
			6&k.10.20.100a & 47,567 & 475 & 5.81 & 68,696 & 665 & 4.65 & 9.12\\
			7&k.10.20.100b & 44,192 & 460 & 6.59 & 64,316 & 644 & 5.55 & 10.11\\
			8&k.10.20.100c & 44,618 & 460 & 6.56 & 64,468 & 644 & 5.53 & 10.14\\
			9&k.10.20.100d & 43,995 & 455 & 7.00 & 73,372 & 728 & 5.37 & 10.70\\
			10&k.10.20.100e & 44,661 & 465 & 6.79 & 65,232 & 651 & 5.50 & 10.73\\\hline
			&Average& 45,006 & 463 & 6.55 & 67,217 & 666 & 5.32 & 10.16\\\hline
			11&k.10.20.150a & 53,168 & 548 & 7.19 & 66,631 & 685 & 6.55 & 10.21\\
			12&k.10.20.150b & 52,828 & 544 & 7.16 & 66,095 & 680 & 6.52 & 9.95\\
			13&k.10.20.150c & 51,866 & 532 & 7.46 & 64,808 & 665 & 6.83 & 10.74\\
			14&k.10.20.150d & 53,372 & 556 & 7.15 & 67,149 & 695 & 6.48 & 10.78\\
			15&k.10.20.150e & 53,047 & 552 & 7.53 & 66,684 & 690 & 6.88 & 10.84\\\hline
			&Average& 52,856 & 546 & 7.30 & 66,273 & 683 &  6.65 & 10.50\\\hline
			16&k.10.20.200a & 51,734 & 549 & 7.91 & 69,173 & 732 & 7.32 & 10.46\\
			17&k.10.20.200b & 52,206 & 555 & 7.51 & 71,248 & 740 & 6.77 & 9.68\\
			18&k.10.20.200c & 51,164 & 540 & 8.20 & 68,586 & 720 & 7.48 & 10.73\\
			19&k.10.20.200d & 50,374 & 540 & 7.86 & 67,524 & 720 & 7.10 & 10.12\\
			20&k.10.20.200e & 52,043 & 555 & 7.88 & 70,003 & 740 & 7.19 & 10.96\\\hline
			&Average& 51,504 & 547 & 7.87 & 69,307 & 730 &  7.17 & 10.39\\\hline
			\bottomrule
		\end{tabular}%
	}}
	\caption{Computational Results Set1 Instances (Run Time: 3,600s)}
	\label{tab:t1}
\end{table}

\begin{table}[htpb]
	\centering
	\scalebox{1.00}{		
	\resizebox{\columnwidth}{!}{%
		\begin{tabular}{|c|c|c|c|c|c|c|c|c|c|}
			\toprule
			& \multicolumn{1}{c|}{} & \multicolumn{3}{c|}{SFD} & \multicolumn{3}{c|}{SFD-R} & \multicolumn{1}{c|}{CPLEX} \\
			\hline
			No. & Instance & MIPs & FCuts & \%Gap & MIPs & FCuts & \%Gap & \%Gap \\
			\hline 
			1&k.10.30.50a&22,483&482&4.04&22,542&491&3.99&4.50\\
			2&k.10.30.50b&20,340&433&6.11&21,935&466&5.31&4.59\\
			3&k.10.30.50c&24,807&529&4.29&25,086&538&3.96&5.84\\
			4&k.10.30.50d&24,432&563&4.10&27,025&592&3.37&5.80\\
			5&k.10.30.50e&24,339&520&3.70&25,009&520&3.04&4.80\\\hline
			&Average&23,280&505&4.45&24,319&521&3.94&5.11\\\hline
			6&k.10.30.100a&27,073&566&4.40&27,452&571&1.04&4.46\\
			7&k.10.30.100b&24,476&624&9.98&26,141&664&6.16&12.40\\
			8&k.10.30.100c&18,711&308&0.86&18,816&313&0.86&6.21\\
			9&k.10.30.100d&26,375&548&6.06&26,663&565&5.97&9.06\\
			10&k.10.30.100e&25,193&550&4.53&26,985&603&2.15&5.55\\\hline
			&Average&24,366&519&5.17&25,211&543&3.24&7.54\\\hline
			11&k.10.30.150a&26,940&739&6.38&27,675&762&5.01&10.78\\
			12&k.10.30.150b&28,677&646&4.66&28,722&651&4.66&8.46\\
			13&k.10.30.150c&24,328&549&7.40&24,333&558&3.65&5.13\\
			14&k.10.30.150d&27,645&611&4.12&29,178&789&1.10&4.54\\
			15&k.10.30.150e&25,312&552&4.28&26,383&582&4.21&12.63\\\hline
			&Average&26,580&619&5.37&27,258&668&3.73&8.31\\\hline
			16&k.10.30.200a&25,817&592&5.21&26,833&610&2.87&4.48\\
			17&k.10.30.200b&27,392&605&4.35&27,944&618&3.62&4.45\\
			18&k.10.30.200c&24,211&555&7.50&25,187&584&7.45&9.45\\
			19&k.10.30.200d&25,471&600&5.16&26,097&609&3.78&7.20\\
			20&k.10.30.200e&24,882&570&8.00&25,884&595&7.98&8.78\\\hline
			&Average&25,555&584&6.04&26,389&603&5.14&6.87\\\hline
			\bottomrule
		\end{tabular}%
	}}
	\caption{Computational Results Set2 Instances (Run Time: 7,200s)}
	\label{tab:t2}
\end{table}

We can see from Tables \ref{tab:t1} and \ref{tab:t2} that none of the algorithms is able to solve any single test instance to optimality within the allotted time, an indication of the difficulty of these instances. However, both the SFD and SFD-R algorithms are able to obtain better bounds than using the direct solver applied to the DEP. In Table \ref{tab:t1} we see that the SFD-R algorithm provides better performance on average over the SFD algorithms in terms of the percentage gap. The gains are much more significant for Set 1 than for Set 2. The results show that incorporating ISG in the SFD algorithm provides gains in gap reduction for both Set 1 and Set 2. This is an indication that reducing the integer set required for generating cuts in an SMIP algorithm can lead to better bounds for SMIP. To see the performance of each method on a given test instance, we plotted the percentage gap versus the instance in Figure \ref{fig:gap}. The graph clearly shows that the SFD-R algorithm provides the best performance overall.
\begin{figure}[H]
	\centering
	\includegraphics[scale=0.55]{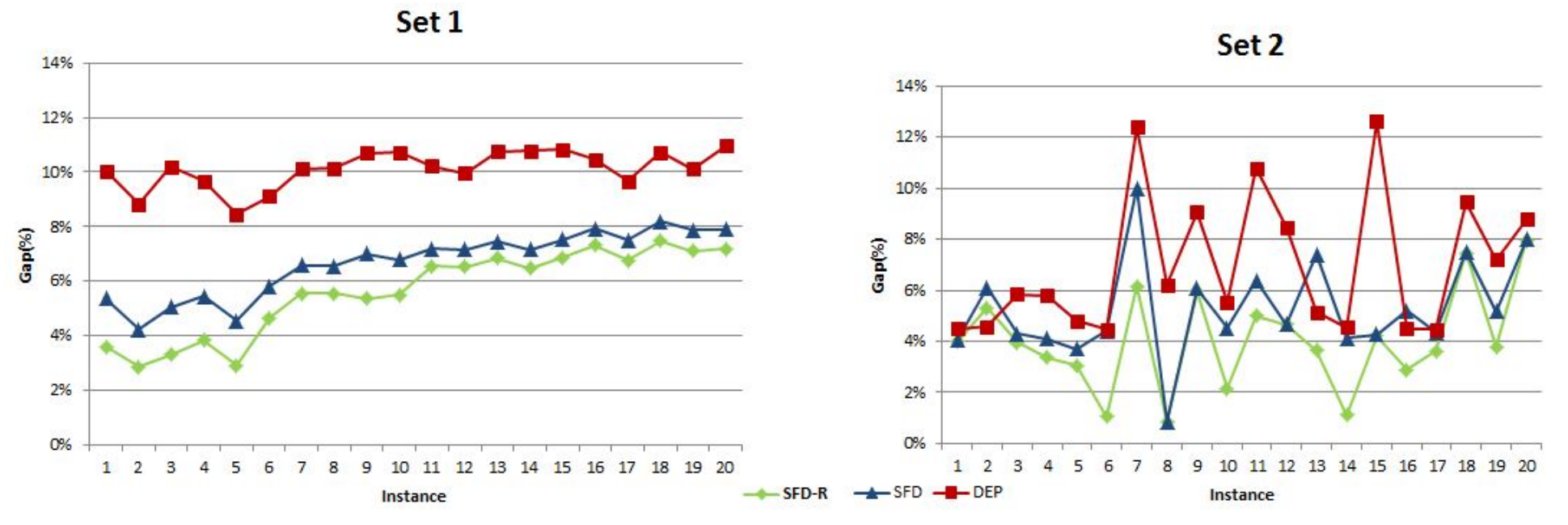}
	\caption{Instances Gap (\%) Runtime (Set 1 - One Hour, Set 2 - Two Hours  )}
	\label{fig:gap}
\end{figure}

We also wanted to look at the number of subproblem MIPs that were solved in the FCG routine under each algorithm. Solving more cut generation MIPs implies fast performance in terms of generating a cut. The results are shown in Figure \ref{fig:mips}, where `\#MIPs' represents the number of MIPs solved using the SFD (without ISG) and SFD-R (with ISG) algorithms. The results clearly show that the SFD-R algorithm solves relatively more MIPs, and thus generates more FD cuts (Figure \ref{fig:fc}), than the SFD algorithm. Finally, we should point out that the larger size test instances (Set 2) generally requires more cuts than the smaller size test instances (Set 1).
\begin{figure}[H]
	\centering
	\includegraphics[scale=0.55]{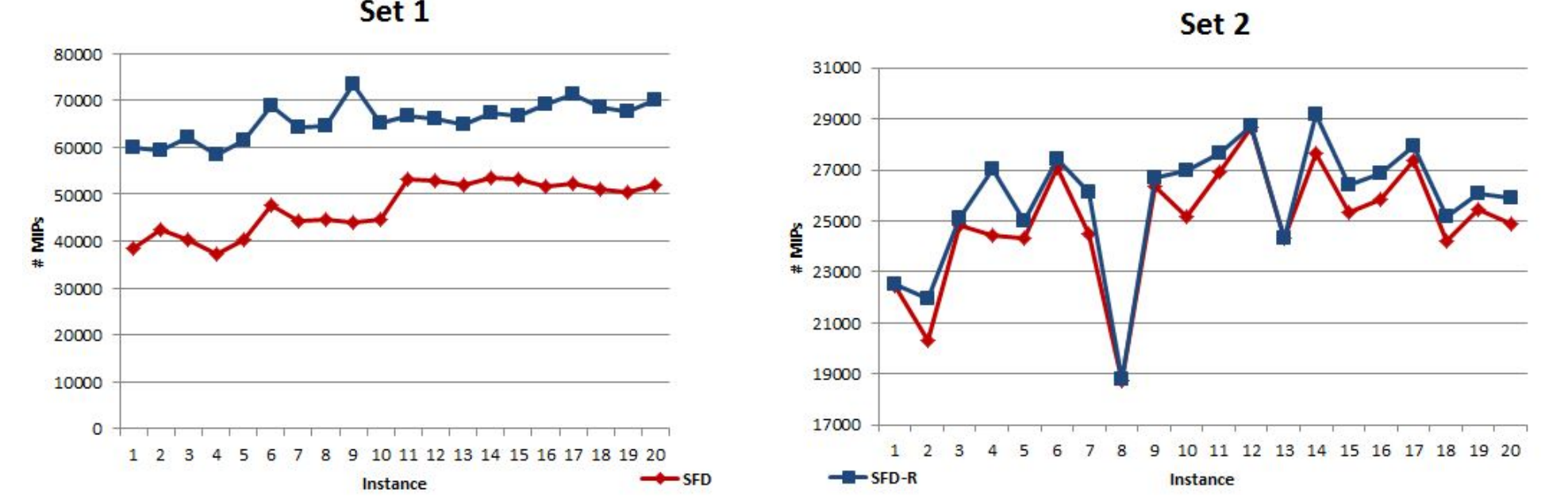}
	\caption{\# MIPs Solved}
	\label{fig:mips}
\end{figure}

\begin{figure}[H]
	\centering
	\includegraphics[scale=0.60]{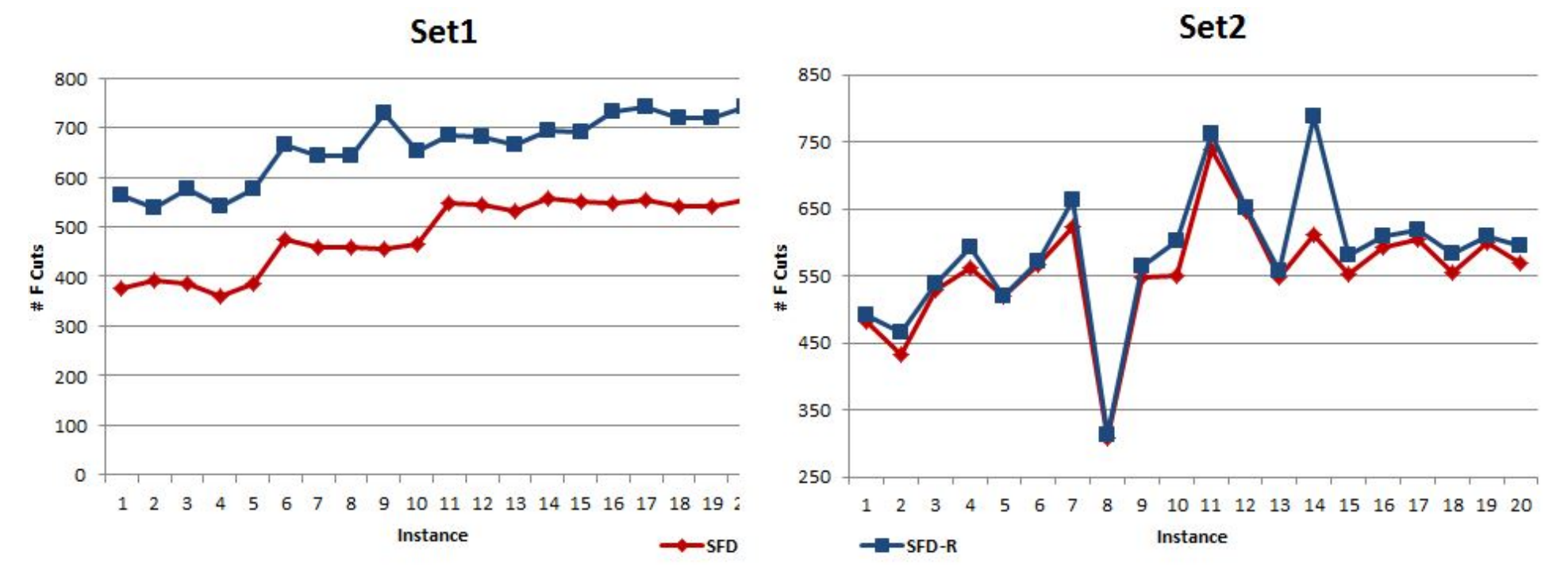}
	\caption{\# Fenchel Cuts}
	\label{fig:fc}
\end{figure}

\section{Conclusion} \label{sec-lconcl}
This work introduces a new integer set reduction procedure for cutting plane methods for SMIP with general integer variables in the second-stage. Example illustrations of the new method in the context of generating Fenchel cutting planes are given. The method is then incorporated into the Fenchel decomposition algorithm for SMIP and a computational study is performed to assess the benefits of the new approach. The results from the computational study show that incorporating integer set reduction in the Fenchel decomposition method leads in having better bounds and provides better performance than a direct solver applied to the deterministic equivalent problem. Also, more cuts are generated in a given time period when integer set reduction is used as opposed to when it is not used. Future work along this line of work include extending the integer set reduction procedure to SMIP with arbitrary or general recourse matrices. Another extension is to incorporate and evaluate the new procedure in other cutting plane methods for SMIP such as disjunctive decomposition and dual decomposition.

\bibliographystyle{plain}
\bibliography{SWFD}

\end{document}